\documentclass[12pt]{amsart}
\usepackage{amssymb,epsfig,amsmath,latexsym,amsthm}
\usepackage{graphicx}
\usepackage{enumitem}
\theoremstyle{plain}
\newtheorem{theorem}{Theorem}[section]
\newtheorem{lemma}[theorem]{Lemma}

\newtheorem{proposition}[theorem]{Proposition}
\newtheorem{corollary}[theorem]{Corollary}

\theoremstyle{definition}

\newtheorem{remark}[theorem]{Remark}

\usepackage{epsfig,color}

\makeatother

\theoremstyle{definition}

\def\fnum{equation}

\numberwithin{equation}{section}

\begin{document}
\title{The catenoid estimate and its geometric applications}

\author{Daniel Ketover}\address{Imperial College London\\Huxley Building, 180 Queen's Gate, London SW7 2RH}
\author{Fernando C. Marques}\address{Department of Mathematics\\Princeton University\\Princeton, NJ 08544}
\author{Andr\'{e} Neves}

\thanks{D.K. was partially supported by NSF-PRF DMS-1401996 as well as by ERC-2011-StG-278940. F.C.M. was partially supported by NSF-DMS 1509027. A.N. was partially supported by ERC-2011-StG-278940 as well as an EPSRC Programme Grant EP/K00865X/1.}

 \email{d.ketover@imperial.ac.uk}
\email{coda@math.princeton.edu}
\email{a.neves@imperial.ac.uk}
\maketitle
\begin{abstract}
We prove a sharp area estimate for catenoids that allows us to rule out the phenomenon of multiplicity in min-max theory in several settings.  We apply it to prove that i) the width of a three-manifold with positive Ricci curvature is realized by an orientable minimal surface ii) minimal genus Heegaard surfaces in such manifolds can be isotoped to be minimal and iii) the ``doublings'' of the Clifford torus by Kapouleas-Yang can be constructed variationally by an equivariant min-max procedure.   In higher dimensions we also prove that the width of manifolds with positive Ricci curvature is achieved by an index $1$ orientable minimal hypersurface.
\end{abstract}

\section{Introduction}
A central difficulty in min-max theory is the phenomenon of multiplicity.   Namely, it may happen that the min-max limit associated to one homotopy class in the space of surfaces is simply an integer multiple of another.  Thus even as one might find distinct homotopy classes in the space of surfaces, one might not be producing geometrically distinct critical points of the area functional.   This issue is already present in the problem of finding closed geodesics.  

In this paper we use a sharp area estimate for catenoids that allows us to exclude multiplicity in several cases.  The ``catenoid estimate" asserts that the area of the unstable catenoid joining two parallel circles in $\mathbb{R}^3$ exceeds the area of the two flat disks subquadratically in the separation between the circles. This enables us in arbitrary three-manifolds to symmetrically foliate a neighborhood around an unstable minimal surface by surfaces with areas strictly less than twice the central minimal surface.  The idea is to take the boundary of a tubular neighborhood of the minimal surface with a small neck attached, which we then ``open up.''  

Indeed, for a closed embedded unstable minimal surface $\Sigma$ embedded in a three-manifold, denote by $T_\epsilon\Sigma$ the $\epsilon$-tubular neighborhood about $\Sigma$.  As long as $\epsilon$ is sufficiently small, $\partial (T_\epsilon\Sigma)$ is diffeomorphic to two disjoint copies of $\Sigma$ if $\Sigma$ is orientable.  If the ambient manifold has positive Ricci curvature, say, then the area of the surfaces comprising $\partial (T_\epsilon\Sigma)$ is of order $\epsilon^2$ below $2|\Sigma|$.   By adding a neck to $\partial (T_\epsilon\Sigma)$  and ``opening it up," the point is that we add area at most on the order $\epsilon^2/(-\log\epsilon)$ before areas start to go down, and thus the areas in our family of surfaces sweeping out $T_\epsilon\Sigma$ have areas strictly below $2|\Sigma|$.  If the ambient manifold does not have positive Ricci curvature, we can instead use the first eigenfunction of the Jacobi operator to produce our parallel surfaces, and the same result applies.  The existence of such explicit sweepouts is what allows us to rule out multiplicity in several cases.  

In higher dimensions analogous results hold but the catenoid is not needed in the construction.  Indeed, if $\Sigma^n$ is an $n$-dimensional minimal hypersurface in $M^{n+1}$, then $\partial (T_\epsilon\Sigma)$ has $n$-dimensional area of order $\epsilon^2$ below $2\mathcal{H}^n(\Sigma^n)$.  But by gluing in a small cylinder around a point $p\in\Sigma^n$ connecting the two components of $\partial (T_\epsilon\Sigma^n)$ and removing the two $n$-balls from $\partial (T_\epsilon\Sigma^n)$, one only needs to add area of order $\epsilon^n$ before the areas start to go down.  Since $\epsilon^n$ is of smaller order than $\epsilon^2$, when ``opening up the hole" one can make a sweepout with all areas strictly less than $2\mathcal{H}^n(\Sigma^n)$.   Dimension two is the ``critical dimension" when more is needed than the standard area comparison argument.

One basic application of the catenoid estimate is that we can rule out one-parameter min-max sequences from collapsing with multiplicity $2$ to a non-orientable minimal surface:
\begin{theorem}\label{onesided}
If $M$ is a three-manifold with positive Ricci curvature, and $\Gamma$ a Heegaard surface realizing the Heegaard genus of $M$, then $\Gamma$ is isotopic to an embedded minimal surface $\Sigma$ of index $1$.  Moreover, $\Sigma$ is the min-max limit obtained via Heegaard sweepouts of $M$ determined by $\Gamma$.
\end{theorem}
\begin{remark}
Theorem \ref{onesided} was proved in \cite{MN} (Theorem 3.4) under the additional assumption that $M$ contains no non-orientable embedded minimal surfaces. 
\end{remark}
In particular, we obtain:
\begin{corollary}
$\mathbb{RP}^3$ endowed with a metric of positive Ricci curvature admits a minimal embedded index $1$ torus.
\end{corollary}

\begin{remark}
Some curvature assumption is necessary in Theorem \ref{onesided}.  Ritor\'{e}-Ros \cite{RR} have shown that there are flat three-tori not admitting index $1$ minimal surfaces of genus three.  The min-max minimal surface coming from a genus three Heegaard splitting must then degenerate to a union of tori in such manifolds.  In other words, a $3$-manifold need not contain a minimal Heegaard surface realizing its Heegaard genus.   
\end{remark}

Similarly, we have:
\begin{theorem}\label{almgrenpitts}
For $3\leq (n+1) \leq 7$, the Almgren-Pitts width (with $\mathbb{Z}$ or $\mathbb{Z}_2$ coefficients) of an orientable  $(n+1)$-manifold with positive Ricci curvature is achieved by an orientable index $1$ minimal hypersurface with multiplicity $1$.
\end{theorem}
The catenoid estimate is needed when $n=2$.
\begin{remark}
Zhou \cite{Z} had proved that the width is achieved by either an orientable index $1$ minimal hypersurface of multiplicity one or a double cover of a non-orientable minimal hypersurface with multiplicity two (see also Mazet-Rosenberg \cite{MR} for an analysis of the least area minimal hypersurface). Theorem \ref{almgrenpitts} rules the second possibility out.
\end{remark}
The classical one-parameter setting for min-max theory is that of sweepouts $\{\Sigma_t\}_{t\in [0,1]}$ such that $\Sigma_0=\Sigma_1=0$.   If one considers the Almgren-Pitts theory with $\mathbb{Z}_2$ coefficients and allows periodic sweepouts, then the analogous statement to Theorem \ref{almgrenpitts} is false.  In round $\mathbb{RP}^3$, for instance,  by rotating an $\mathbb{RP}^2$ one can produce a non-trivial cycle comprised of minimal projective planes of area $2\pi$, each realizing the $\mathbb{Z}_2$ Almgren-Pitts width for that kind of sweepout.

The catenoid estimate and the equivariant min-max theory of \cite{Ke2}  allow one also to give a min-max construction of ``doublings" of minimal surfaces.  We give a new construction and variational interpretation of the Kapouleas-Yang \cite{KY} doublings of the Clifford torus in the round three-sphere that was first proposed in 1988 by Pitts-Rubinstein \cite{PR1}. Pitts-Rubinstein also noted (see Remark 4 in \cite{PR1}) that results analogous to our catenoid estimate would be necessary to carry out the construction. Indeed, the main issue is that when running  an equivariant min-max procedure one must rule out the min-max sequence from collapsing with multiplicity two to the Clifford torus. The explicit sweepout produced by the catenoid estimate stays below in area twice the Clifford torus and thus we can obtain:
\begin{theorem}\label{doubledtori}
For each $g\geq 2$, there exists a closed embedded minimal surface $\Sigma_g$ resembling a doubled Clifford torus in $\mathbb{S}^3$.  The area of $\Sigma_g$ is strictly less than $4\pi^2$ (twice the area of the Clifford torus $C$).  Moreover $\Sigma_g\rightarrow 2C$ in the sense of varifolds as $g\rightarrow\infty$ and the genus of $\Sigma_g$ also approaches infinity.  The surfaces $\Sigma_g$ arise as min-max limits for a suitable equivariant saturation of sweepouts of $\mathbb{S}^3$.
\end{theorem}

\begin{remark}
In fact, for $g$ large enough, one can show that the genus of $\Sigma_g$ is $g^2+1$ and the surfaces consist of two tori parallel to the Clifford torus joined by $g^2$ necks placed symmetrically along a grid.  See \cite{Ke2} for more details.
\end{remark}

\begin{remark}
The catenoid estimate suggests heuristically that unstable minimal surfaces should not arise with multiplicity in the min-max theory.   This has been explicitly conjectured for generic metrics by the second and third named authors in \cite{MN4}, where they have confirmed it in the Almgren-Pitts setting when the number of parameters is one.  We know of no instance where an unstable component occurs with multiplicity except in the case of geodesics (see \cite{A}).
\end{remark}

The organization of this paper is as follows.  In Section 2 we prove the catenoid estimate, first in $\mathbb{R}^3$, then in an arbitrary three-manifold. In Section 3 we give the applications.  In Section 4 we prove Theorem \ref{almgrenpitts} in higher dimensions.
\section{Catenoid estimate}
\subsection{Catenoid estimate in $\mathbb{R}^3$}
First we explain the catenoid estimate in $\mathbb{R}^3$.  We will not need the results of this section for the generalization to arbitrary three-manifolds or in the rest of the paper, but we include it since it is the motivating heuristic. 

Consider the two parallel circles in $\mathbb{R}^3$:
\begin{align}
&C_1(r,h):= \{(h,y,z)\;|\; y^2+z^2=r^2\}\\
&C_2(r,h):= \{(-h,y,z)\;|\; y^2+z^2=r^2\}
\end{align}
As long as $h$ is small compared to $r$, there are two stable minimal surfaces with boundary $C_1(r,h)\cup C_2(r,h)$.  The first is the surface $S_1(r,h)$ consisting of the union of two stable flat disks, with area $2\pi r^2$.  The second minimal surface $S_2(r,h)$ is the stable catenoid.  As $h$ tends to zero, the area of $S_2(r,h)$ converges to $0$.   \\ \indent
By analogy with Morse theory, given the two stable minimal surfaces with the same boundary, one expects to find an unstable minimal surface between them.  One can consider sweepouts by annuli of the region between the stable catenoid and two disks (degenerating to two disks at one side).   Running a min-max procedure, one finds the unstable catenoid  $U(r,h)$ as the surface realizing the width for this family.  The catenoid estimate concerns the area of $U(r,h)$ when $h$ is small.   Here and throughout this paper, if $\Sigma$ is a set in $\mathbb{R}^3$ or in a $3$-manifold, $|\Sigma|$ denotes the 2-dimensional Hausdorff measure of $\Sigma$.

\begin{proposition} (Catenoid estimate in $\mathbb{R}^3$)
For $r>0$ there exists $h(r)>0$ so that if $h<h(r)$ then
\begin{equation}\label{catenoidestimate}
|U(r,h)|\leq 2\pi r^2+\frac{4\pi h^2}{(-\log h)}.
\end{equation}
\end{proposition}
\begin{remark}
For any $\delta>0$, by taking $h(r)$ even smaller one can replace $4\pi$ on the RHS of \eqref{catenoidestimate} by $2\pi(1+\delta)$.
\end{remark}
\begin{remark}
One can construct by hand a sweepout $\{\Sigma_t\}_{t=0}^r$ of annuli foliating the region between the stable catenoid and two disks by simply cutting out at time $t$ a disk of radius $t$ from each component $S_1(r,h)$ and gluing in a cylinder.  The maximum area of a surface in this sweepout is of order $h^2$ above the area of the two disks.  The point of the catenoid estimate \eqref{catenoidestimate} is that this sweepout is far from efficient - using an optimal sweepout the maximal area is of order $h^2/(-\log h)$ above the area of the two disks.
\end{remark}


\begin{proof}
The unstable catenoid with boundary $C_1(r,h)\cup C_2(r,h)$ is obtained by rotating the function $f(x)=c\cosh(x/c)$ about the $x$ axis.  The constant $c=c(r,h)$ is the smaller solution to
\begin{equation}\label{param}
r=c\cosh(h/c).
\end{equation} 
Using the identities for hyperbolic trigonometric functions, one can obtain a formula for the area of $U(r,h)$:
\begin{align}
|U(r,h)| &=\int_{-h}^{h}2\pi f(x)\sqrt{1+(f'(x)^2)}dx\nonumber\\
&=\int_{-h}^{h}2\pi c\cosh^2(x/c)dx\nonumber\\
&=\pi c^2\sinh(2h/c)+2\pi hc\nonumber\\
&=2\pi c^2 \cosh(h/c)\sinh(h/c)  +2\pi hc\nonumber \\
&=2\pi r\sqrt{r^2-c^2}+2\pi hc \mbox{ , using \eqref{param}}\nonumber\\ 
&\leq 2\pi r^2+2\pi hc.\label{big2}
\end{align}

It remains to determine how $c(r,h)$ depends on $h$ when $h$ is small.  Since $c$ is defined as the smaller solution to \eqref{param}, it follows that 
\begin{equation}\label{c}
\lim_{h\rightarrow 0} h/c=\infty.
\end{equation}
To see \eqref{c}, it is convenient to change variables in \eqref{param} to $x:=h/c$ and $\lambda:=r/h$.  Then \eqref{c} is equivalent to the claim that the larger solution of 
\begin{equation}\label{ref}
\lambda x = \cosh (x),
\end{equation}
tends to infinity as $\lambda\rightarrow\infty$.  At the smaller solution to \eqref{ref} near $0$, the derivative of $\cosh(x)$ is smaller than that of $\lambda x$.  The second larger solution to \eqref{ref} must occur at an $x$ larger than the $x_0$ at which the derivatives of $\lambda x$ and $\cosh (x)$ are equal, i.e., at the $x_0$ satisfying $\lambda = \sinh (x_0)$.  But $x_0$ approaches infinity as $\lambda\rightarrow\infty$.  Thus \eqref{c} is established.

 There exists a function $g(x)$ defined for $x>0$ so that: 
\begin{equation}\label{approx}
 \log(\cosh(x))=x-\log2+g(x) \mbox{ where } \lim_{x\rightarrow\infty}g(x)=0. 
\end{equation}
Taking the logarithm of \eqref{param} and applying \eqref{approx} we obtain
\begin{equation}\label{basic}
\log r - \log c=h/c-\log2+g(h/c).
\end{equation}
Rearranging \eqref{basic} we conclude 
\begin{equation}\label{rearranged}
\log 2r + \log (h/c)-\log h=h/c+g(h/c).
\end{equation}
Dividing \eqref{rearranged} by $h/c$ we obtain,
\begin{equation}\label{final}
\log 2r/(h/c)+ \log (h/c)/(h/c)-c(\log h)/h=1+(c/h)g(h/c).
\end{equation}
Using \eqref{c} we then conclude from \eqref{final}
\begin{equation}
\lim_{h\rightarrow 0} \frac{c(-\log h)}{h}=1, 
\end{equation}
from which we see that if $h$  is small enough, 
\begin{equation}\label{cestimate}
c(h,r)\leq\frac{2h}{(-\log h)}.
\end{equation}  
Plugging \eqref{cestimate} into \eqref{big2} we obtain \eqref{catenoidestimate}.
\end{proof}

\subsection{Catenoid estimate in a $3$-manifold}\label{catenoidinthree}
The catenoid estimate in a $3$-manifold asserts that we can construct a sweepout interpolating between the boundary of a tubular neighborhood about an unstable minimal surface and a graph on the minimal surface where each surface in the sweepout has area less than twice that of the central minimal surface. We will not explicitly need to use catenoids to construct our sweepout; instead we use logarithmically cut off parallel surfaces that turn out to have areas of the correct order predicted by the catenoid estimate.  The catenoid estimate can thus be interpreted as yet another instance of the ``logarithmic cutoff trick."   

One technical difficulty that makes our arguments slightly more involved than those used in the standard log cutoff trick is that we are considering a one-parameter family of normal graphs whose gradients are becoming singular at a point and we need all estimates uniform in this family. 

Let us first introduce the notion of a continuous sweepout that we will use in this paper.   Set $I^n=[0,1]^n$ and let $\{\Sigma_t\}_{t\in I^n}$ be a family of closed subsets of $M$ and $B\subset\partial I^n$.   We call the family $\{\Sigma_t\}$ a \emph{(genus-g) sweepout} if

\begin{enumerate}
\item $\mathcal{H}^2(\Sigma_t)$ is a continuous function of $t\in I^n$,
\item $\Sigma_t$ converges to $\Sigma_{t_0}$ in the Hausdorff topology as $t\rightarrow t_0$.
\item For $t_0\in I^n\setminus B$, $\Sigma_{t_0}$ is a smooth embedded closed surface of genus $g$ and $\Sigma_t$ varies smoothly for $t$ near $t_0$.
\item For $t\in B$, $\Sigma_t$ consists of the union of a $1$-complex together (possibly) with a smooth surface.
\end{enumerate}

Loosely speaking, a sweepout consists of genus $g$ surfaces varying smoothly, that could degenerate to $1$-d graphs or pieces of graphs together with smooth surfaces at the boundary of the parameter space.  

For a closed embedded surface $\Sigma\subset M$, $\phi>0$ a smooth function defined on $\Sigma$ and $\epsilon>0$, set the $\phi$-adapted tubular neighborhood about $\Sigma$ to be:  
\begin{equation}
T_{\epsilon\phi}(\Sigma):=\{\exp_p(s\phi(p)N(p)): s\in [-\epsilon,\epsilon],p\in\Sigma\},
\end{equation}
where $N$ is a choice of unit normal on $\Sigma$.  

Given distinct points $p_1,...,p_k\in\Sigma$ and $1$-d graph $\mathcal{G}\subset\Sigma$, let us say that \emph{$\Sigma\setminus\{p_1,...p_k\}$ retracts onto $\mathcal{G}$} if for any $\epsilon>0$ small enough there exists a smooth deformation retraction $\{R_t\}_{t=0}^1$ of $M\setminus\cup_{i=1}^kB_{\epsilon}(p_i)$ onto $\mathcal{G}$ such that 
\begin{equation}\label{complementofballs}
R_t:M\setminus\cup_{i=1}^kB_{\epsilon}(p_i)\rightarrow M\setminus\cup_{i=1}^kB_{\epsilon}(p_i).
\end{equation}
We further require that $R_t$ is one-to-one for $t\neq 1$.



 For any surface $\Sigma$ of genus $g$ and $p\in\Sigma$, for instance, there always exists a retraction from $\Sigma\setminus\{p\}$ onto a wedge of $2g$ circles.   If $\Sigma$ is a two-sphere and $p_1, p_2\in\Sigma$ are distinct, then there is similarly a retraction from $\Sigma\setminus\{p_1,p_2\}$ onto a closed circle.

With these definitions we can now state the catenoid estimate in a general three-manifold:


\begin{theorem} (Catenoid estimate)\label{catenoidinthreemanifold}
Let $M$ be a $3$-manifold and let $\Sigma$ be a closed orientable unstable embedded minimal surface of genus $g$ in $M$.  Denote by $\phi$ the lowest eigenfunction of the Jacobi operator on $\Sigma$ normalized so that $\|\phi\|_{L^2}=1$ and let $N$ be a choice of unit normal on $\Sigma$. Fix $p_1,p_2...p_k\in\Sigma$ and a graph $\mathcal{G}$ in $\Sigma$ so that $\Sigma\setminus\{p_1,...p_k\}$ retracts onto $\mathcal{G}$.  Then there exists $\epsilon_0>0$ and $\tau>0$ so that whenever $\epsilon\leq\epsilon_0$ there exists a sweepout $\{\Lambda_t\}_{t=0}^1$ of $T_{\epsilon\phi}(\Sigma)$ so that:
\begin{enumerate}
\item $\Lambda_t$ is a smooth surface of genus $2g+k-1$ (i.e. two copies of $\Sigma$ joined by $k$ necks)  for $t\in (0,1)$,
\item $\Lambda_0=\partial (T_{\epsilon\phi}(\Sigma))\cup\bigcup_{i=1}^k\{\exp_{p_i}(s\phi(p_i)N(p_i)): s\in [-\epsilon,\epsilon]\}$
\item $\Lambda_1=\mathcal{G}$
\item For all $t\in [0,1]$, $|\Lambda_t|\leq 2|\Sigma|-\tau\epsilon^2$.
\end{enumerate}  
If $M$ has positive Ricci curvature, we can set in the above $\phi=1$ even though it might not be an eigenfunction.
\end{theorem}


In order to facilitate computations for areas of normal graphs, we will use Fermi coordinates, which are essentially normal coordinates adapted to the tubular neighborhood of a submanifold.  We follow the exposition in Section 4 of \cite{PS}.

\subsection{Fermi coordinates}

Denote by $\Lambda$ a smooth closed embedded (not necessarily minimal) hypersurface in a Riemannian $(n+1)$-manifold $M$. For $z$ small we can consider the parallel hypersurfaces:
\begin{equation}
\Lambda_z=\{\exp_p(zN(p)) : p\in\Lambda\},
\end{equation}
where $N$ is a choice of unit normal vector field on $\Lambda$.
The map
\begin{equation}
F(p,z) := \exp_p(z N(p))
\end{equation}
is a diffeomorphism (for $\epsilon$ small enough) from a neighborhood of $(p,0)\subset\Lambda\times\mathbb{R}$ into a neighborhood of $p$ in $T_{\epsilon}(\Lambda)\subset M$.

Denote by $g_z$ the induced metric on the surface $\Lambda_z$.    By Gauss' lemma, the metric $g$ on $M$ has a product expansion in Fermi coordinates:
\begin{equation}\label{gauss}
g = g_z+dz^2
\end{equation}

Let $\phi$ be a smooth function defined on $\Lambda_0=\Lambda$.   We can then consider surfaces that are normal graphs over $\Lambda$:
\begin{equation}
\Lambda_\phi=\{\exp_p(\phi(p)N(p)) : p\in\Lambda\}
\end{equation}
By \eqref{gauss}, the induced metric on $\Lambda_\phi$ is given by 
\begin{equation}
g_{\Lambda_\phi} = g_\phi + d\phi\otimes d\phi.
\end{equation}

One can then compute the area of the surfaces $\Lambda_\phi$:
\begin{equation}\label{graph}
|\Lambda_\phi| = \int_\Lambda \sqrt{1+ |\nabla^{g_\phi}\phi|^2_{g_\phi}}\text{dv}_{g_\phi}.
\end{equation}

In the remainder of this section, we will apply \eqref{graph} to the function $h\phi$ to obtain an expansion in $h$ (for $h$ small) for the area of the normal graph determined by $h\phi$ in terms of quantities defined on $\Lambda$.  We prove the following:

\begin{proposition}\label{error}
If $\Lambda$ is an open set contained in a minimal hypersurface, and $\phi$ a smooth function defined on $\Lambda$, then there exists $h_0>0$ so that for $h\leq h_0$ we have the expansion
\begin{equation}\label{witherror}
|\Lambda_{h\phi}|\leq |\Lambda|+\frac{h^2}{2}\int_\Lambda (|\nabla\phi|^2-\phi^2(|A|^2+\text{Ric}(N,N)))+Ch^3\int_\Lambda (1+|\nabla\phi|^2), 
\end{equation}
where $h_0$ and $C$ depend only on $\Lambda$ and $\|\phi\|_{L^\infty}$.
\end{proposition}

\begin{remark}
The estimate \eqref{witherror} is sharp.  Indeed, it is well known that if $\Lambda$ is a minimal surface, then from the Taylor expansion of the area functional we always have:
\begin{equation}\label{simple}
|\Lambda_{h\phi}|\leq|\Lambda|-\frac{h^2}{2}\int_\Lambda\phi L\phi+\mathcal{O}(h^3),
\end{equation}
where $L$ denotes the Jacobi operator on a minimal surface,
\begin{equation}\label{jacobi}
L=\Delta + |A|^2+\text{Ric}(N, N).
\end{equation} 

The reason \eqref{simple} by itself is not sufficient for our purposes is that we will be considering families of functions $\phi_t$ whose gradients are blowing up at a point as $t\rightarrow 0$ and we need to ensure that the $\mathcal{O}(h^3)$ term in \eqref{simple} is bounded independently of $t$.  For the family of functions we will consider, the $L^2$ norm of the gradients will be uniformly bounded as well as their $L^\infty$ norms, and thus from \eqref{witherror} we will in fact obtain \eqref{simple}.
\end{remark}

In order to expand \eqref{graph}, we will need the expansion for the metrics $g_z$ on $\Lambda_z$ in terms of $g_0$, the induced metric on $\Lambda_0=\Lambda$.  Such an expansion is derived in Pacard-Sun \cite{PS}:

\begin{lemma} (Proposition 5.1 in \cite{PS})
\begin{equation}\label{expansion}
g_z=g_0-2zA+z^2T + \mathcal{O}(z^3),
\end{equation}
where $A$ denotes the second fundamental form on $\Lambda$ and the tensor $T$ is defined by
\begin{equation}\label{pacard}
T =A\otimes A+g(R(N,.)N,.), 
\end{equation}
where $(A\otimes A)(v_1,v_2)=g_0(\nabla_{v_1} N, \nabla_{v_2} N)$ for $v_1$ and $v_2$ in $T\Lambda$.
\end{lemma}

\emph{Proof of Proposition \ref{error}:}
We need to expand both the integrand and volume element in \eqref{graph}.  We will first handle the integrand. Recall the Neumann formula for matrix inversion of perturbations: If $\tilde{g}$ is a square matrix with expansion:
\begin{equation}
\tilde{g} = g + \epsilon X + \epsilon^2 Y + \mathcal{O}(\epsilon^3),
\end{equation}
 then the inverse of $\tilde{g}$ can be expressed as:
\begin{equation}\label{forminverse}
\tilde{g}^{-1}= g^{-1}-\epsilon(g^{-1}Xg^{-1})+ \epsilon^2((g^{-1}X)^2g^{-1}-g^{-1}Yg^{-1})+\mathcal{O}(\epsilon^3)
\end{equation}
Using the expansion \eqref{expansion} in \eqref{forminverse} (setting $X=-2A$, $Y=T$ and $\epsilon=h\phi$) we can express the inverse of $g_{h\phi}$ by
\begin{equation}
(g_{h\phi})^{-1}=g^{-1}_0+2h\phi (g^{-1}_0Ag^{-1}_0)+4h^2\phi^2((g^{-1}_0A)^2g^{-1}_0-g^{-1}_0Tg^{-1}_0)+\mathcal{O}(h^3\phi^3).
\end{equation}
Thus we can expand $|\nabla(h\phi)|^2_{g_{h\phi}}=h^2((g_{h\phi})^{-1})^{ij}\phi_i\phi_j$ (where there is summation in $i$ and $j$):
\begin{equation}
|\nabla(h\phi)|^2_{g_{h\phi}}=h^2|\nabla\phi|^2_{g_0}+2h^3\phi A(\nabla\phi, \nabla\phi)+\mathcal{O}(h^4\phi^2|\nabla\phi|_{g_0}^2)
\end{equation}

In other words we obtain
\begin{equation}\label{integrand}
\sqrt{|\nabla(h\phi)|^2_{g_{h\phi}}+1}\leq\sqrt{1+h^2|\nabla\phi|^2_{g_0}(1 + Ch)}.
\end{equation}
where the expansion \eqref{integrand} holds for $h$ sufficiently small (depending only on $\|\phi\|_{L^\infty}$) and where $C$ also depends only on $\|\phi\|_{L^\infty}$ and $\Lambda$.  This gives a bound for the integrand that we need to estimate in \eqref{graph}.  

We must also compute the expansion for the volume element $\text{dv}_{g_{h\phi}}$. Recall that if 
\begin{equation}
\tilde{g} = g + \epsilon X + \epsilon^2 Y+ \mathcal{O}(\epsilon^3),
\end{equation}
then one has the following expansion for $\det(\tilde{g})$:
\begin{equation}\label{detexp}
\det(\tilde{g})=\det(g)(1+\epsilon \text{tr}(g^{-1}X)+\epsilon^2(\text{tr}(g^{-1}Y)+\text{tr}_2 (g^{-1}X))+\mathcal{O}(\epsilon^3),
\end{equation}
where $\text{tr}$ denotes trace, and $\text{tr}_2(M)=\frac{1}{2}((\text{tr}(M))^2-\text{tr}(M^2))$ (i.e. the sum over all products of two different eigenvalues.   

Plugging our expansion \eqref{expansion} into \eqref{detexp}, using also that $\text{tr}(g^{-1} A)=0$ by the minimality of $\Lambda$, that $\text{tr}(g^{-1}T)=|A|^2-\text{Ric}(N, N)$ (by \eqref{pacard}) and that $\text{tr}_2(g^{-1}A)=-\frac12|A|^2$ (again by minimality), we conclude that
\begin{equation}\label{det}
\det{g_{h\phi}}=(\det{g_0})(1-(h\phi)^2(|A|^2+\text{Ric}(N,N))+\mathcal{O}(\phi^3h^3)).
\end{equation}
Substituting \eqref{det} and \eqref{integrand} into the formula for the area of a graph \eqref{graph} in Fermi coordinates, we obtain
\begin{equation}
|\Lambda_{h\phi}|\leq\int_\Lambda\sqrt{1+h^2(|\nabla\phi|^2_{g_0}-\phi^2|A|^2-\phi^2\text{Ric}(N,N))+Ch^3(1+|\nabla\phi|^2_{g_0})}\text{dv}_\Lambda.
\end{equation}
Using the inequality $\sqrt{1+x}\leq 1+x/2$, for $x\geq -1$ to estimate the integrand, we then obtain \eqref{witherror}.\qed


\subsection{Proof of Catenoid Estimate (Theorem \ref{catenoidinthreemanifold})}
In the following, $C$ will denote a constant potentially increasing from line to line but only depending on the geometry of $\Sigma$ and $\|\phi\|_{L^\infty}$.  All balls $B_r(x)$ will be ambient metric balls.
\begin{proof} 
Let us assume that $k=1$ and set $p:=p_1$ so that $\Sigma\setminus\{p\}$ retracts onto the given graph $\mathcal{G}$.  The general case will then follow with trivial modifications.
 
Choose $R>0$ sufficiently small so that $\Sigma\cap B_t(p)$ is a disk for all $t\in (0,R]$ and so that there exists $D>0$ such that for any $t\leq R$ there holds
\begin{equation}\label{mono}
|\Sigma\cap \partial B_t(p)|\leq Dt.
\end{equation}

For any $x\in M$, set $r(x):=\text{dist}_M(x,p)$. Let $-\lambda<0$ be the lowest negative eigenvalue of the Jacobi operator \eqref{jacobi} corresponding to the eigenfunction $\phi$ (so that $L\phi=\lambda\phi$). We define for $t\in [0,R]$, the logarithmically cut-off functions:
\begin{equation*}
    \eta_t(x) = \begin{cases}
               1      & r(x)\geq t\\
               (1/\log(t))(\log t^2-\log r(x))          & t^2\leq r(x)\leq t\\
               0 & r(x)\leq t^2
           \end{cases}
\end{equation*}
Then set $\phi_t(x)=\phi(x)\eta_t(x)$.
For each $t\geq 0$, we consider the parallel surfaces
\begin{equation}
\Lambda'_{h,\pm t}=\{\exp_p(\pm h\phi_t(p)N) : p\in\Sigma\},
\end{equation}
where $N$ is a choice of unit normal vector on $\Sigma$. 
The surfaces that will make up our foliation, $\Lambda_{h,t}$,  (for $0\leq t\leq R$) are then defined to be:
\begin{equation}
\Lambda_{h,t}:=(\Lambda'_{h,t}\setminus B_{t^2})\cup(\Lambda'_{h,-t}\setminus B_{t^2}),
\end{equation}
where by $B_{t^2}$ we mean $B_{t^2}(p)\cap\Sigma$ and where $h$ and $R$ will be fixed later to be suitably small.  Note that the surfaces $\Lambda_{h,t}$ converge (in the varifold sense) to $\partial(T_h\Sigma)$ as $t\rightarrow 0$.  Because $\phi >0$, the surfaces $\Lambda_{h,t}$ are (piecewise smooth) embedded surfaces.

By applying \eqref{witherror} to $\phi_t$ and on the set $\Sigma\setminus B_t$ we obtain
\begin{align}\label{basictaylor}
|\{\exp_p(h\phi_t(p)N(p)) : p\in\Sigma\setminus B_t\}|\leq &  |\Sigma|-|B_t|\nonumber\\
&+\frac{h^2}{2}\int_{\Sigma\setminus B_t} (|\nabla\phi_t|^2-\phi_t^2(|A|^2+\text{Ric}(N,N)))\nonumber\\
&+Ch^3\int_{\Sigma\setminus B_t} (1+|\nabla\phi_t|^2),
\end{align}
where $C$ only depends on $\|\phi\|_{L^\infty}$ and the geometry of $\Sigma$.  For $R$ sufficiently small, since $\phi$ is an eigenfunction for $L$ with $\|\phi\|_{L^2}=1$ and $\phi=\phi_t$ on $\Sigma\setminus B_t$, we have for all $t\in [0,R]$, that
\begin{equation}\label{eigenvalue}
\int_{\Sigma\setminus B_t} (|\nabla\phi_t|^2-\phi_t^2(|A|^2+\text{Ric}(N,N)))\leq-\frac{\lambda}{2}.
\end{equation}
Also since the gradients of $\phi_t$ are uniformly bounded on $\Sigma\setminus B_t$ in $t$, we obtain that the $h^3$ error terms in \eqref{basictaylor} are also bounded independently of $t$.  In total we obtain
\begin{equation}
|\{\exp_p(h\phi_t(p)N(p)) : p\in\Sigma\setminus B_t\}|\leq |\Sigma|-|B_t|-\frac{\lambda}{4}h^2+Ch^3,
\end{equation}
where $C$ depends on $\phi$ (and not $t$).
Then by shrinking $h$ to absorb the $h^3$ term we obtain for $t\in [0,R]$ and for $h$ sufficiently small,
\begin{equation}\label{big}
|\{\exp_p(h\phi_t(p)N(p)) : p\in\Sigma\setminus B_t\}|\leq |\Sigma|-|B_t|-\frac{\lambda}{8}h^2.
\end{equation}

Let us now apply \eqref{witherror} to estimate the area 
\begin{equation}
|\{\exp_p(h\phi_t(p)N(p)) : p\in B_t\setminus B_{t^2}\}|,
\end{equation}
i.e. the part of the normal graph where the logarithmic cutoff function is present and where the gradient terms are unbounded in $t$. By potentially shrinking $h$ to absorb the $h^3$ terms and shrinking $R$ to bound the terms $|\int_{B_t\setminus B_{t^2}}\phi_t^2(|A|^2+\text{Ric}(N,N))|\leq\frac{\lambda}{32}$ uniformly in $t$ we obtain,
\begin{equation}\label{small}
|\{\exp_p(h\phi_t(p)N(p)) : p\in B_t\setminus B_{t^2}\}|\leq |B_t|-|B_{t^2}|+\frac{\lambda}{16}h^2+h^2\int_{B_t\setminus B_{t^2}}|\nabla\phi_t|^2.
\end{equation}

By adding together \eqref{big} with \eqref{small}, we get for $t\in [0,R]$ and $h$ small enough:
\begin{equation}\label{est2}
|\Lambda'_{h,t}|\leq |\Sigma|-\frac{\lambda}{16} h^2 +h^2\int_{B_t\setminus B_t^2} |\nabla\phi_t|^2.
\end{equation}

We can now apply the logarithmic cutoff trick (see for instance \cite{CS}) to estimate 
the gradient term $\int_{B_t\setminus B_t^2} |\nabla\phi_t|^2$ in \eqref{est2}.

By Cauchy-Schwartz we obtain
\begin{equation}\label{cs}
\int_{B_t\setminus B_t^2} |\nabla\phi_t|^2\leq 2\int_{B_t\setminus B_t^2}(\phi^2|\nabla\eta_t|^2+\eta_t^2|\nabla\phi|^2)\leq 2(\sup\phi)^2\int_{B_t\setminus B_t^2}|\nabla\eta_t|^2+2\int_{B_t\setminus B_t^2}|\nabla\phi|^2.
\end{equation}

Note that on $B_t\setminus B_t^2$,
\begin{equation}
|\nabla\eta_t|^2=\frac{1}{(\log t)^2}\frac{|\nabla r|^2}{r^2}.
\end{equation}

Thus by the co-area formula we obtain
\begin{equation}\label{coarea}
\int_{B_t\setminus B_t^2} |\nabla\eta_t|^2=\frac{1}{(\log t)^2}\int_{t^2}^t\frac{1}{\lambda^2}\int_{r=\lambda} |\nabla r|.
\end{equation}
Since $|\nabla r|\leq 1$, the inner integral in \eqref{coarea} can then be estimated using \eqref{mono}
\begin{equation}\label{gradr}
\int_{r=\lambda} |\nabla r|\leq |\Sigma\cap\partial B_\lambda|\leq D\lambda.
\end{equation}

Plugging \eqref{gradr} into \eqref{coarea} we obtain:
\begin{equation}\label{log}
\int_{B_t\setminus B_t^2} |\nabla\eta_t|^2\leq\frac{D}{(\log t)^2}\int_{t^2}^t\frac{1}{\lambda}d\lambda\leq\frac{D}{-\log t}.
\end{equation}

Thus we obtain combining \eqref{cs} and \eqref{log} and setting $A:=2D(\sup\phi)^2$
\begin{equation}\label{cut}
\int_{B_t\setminus B_t^2} |\nabla\phi_t|^2\leq\frac{A}{(-\log t)}+2\int_{B_t\setminus B_t^2} |\nabla\phi|^2.
\end{equation}

By further shrinking of $R$, we can guarantee that for all $t\leq R$, 
\begin{equation}\label{shrunk}
\int_{B_t\setminus B_t^2}|\nabla\phi|^2\leq\int_{B_R}|\nabla\phi|^2\leq\frac{\lambda}{64}.
\end{equation}

Plugging \eqref{cut} and \eqref{shrunk} back into \eqref{est2} we obtain for all $t\leq R$, 
\begin{equation}\label{last}
|\Lambda'_{h,t}|\leq |\Sigma|-\frac{\lambda}{32} h^2+\frac{A}{(-\log t)}h^2.
\end{equation}
Thus adding together contributions from the two components of $\Lambda_{h,t}$ we have 
\begin{equation}\label{areaoffoliation}
|\Lambda_{h,t}|\leq 2|\Sigma|-\frac{\lambda}{16} h^2-2|B_{t^2}| +\frac{A}{(-\log t)}2h^2.
\end{equation}
Note that for $t$ sufficiently small  we see from \eqref{areaoffoliation} that $|\Lambda_{h,t}|<2|\Sigma|$ (since $1/(-\log t)\rightarrow 0$ as $t\rightarrow 0$), which is what we needed.  However, as $t$ increases up to $R$, the two terms in \eqref{areaoffoliation} of order $h^2$ become comparable and thus we need to further shrink $R$ so that 
\begin{equation}
\frac{2A}{-\log R} < \frac{\lambda}{32},
\end{equation} 
to ensure that 
\begin{equation}\label{controlledfoliation}
|\Lambda_{h,t}|\leq 2|\Sigma|-\frac{\lambda}{32} h^2 - 2|B_{t^2}| < 2|\Sigma|-\frac{\lambda}{32} h^2,
\end{equation} 
for all $t\in [0,R]$.  

The parameter $R$ will now be fixed.  From \eqref{controlledfoliation} we also obtain that
\begin{equation}\label{openhole}
|\Lambda_{h,R}|<2|\Sigma|-2|B_{R^2}|.
\end{equation}
The estimate \eqref{openhole} guarantees that ``opening the hole" up to $t=R$ drops area by a definite amount (depending on $R$ and not $h$).

As varifolds, $\Lambda_{h,R}$ converges to $\Sigma\setminus B_{R^2}$ with multiplicity $2$ as $h\rightarrow 0$.  Thus we have
\begin{equation}\label{areaclose}
2|\Sigma\setminus B_{R^2}|\leq |\Lambda_{h,R}|+\epsilon(h),
\end{equation}
where $\epsilon(h)\rightarrow 0$ as $h\rightarrow 0$.
  By assumption $\Sigma\setminus B_{R^2}$ retracts to the graph $\mathcal{G}$ so that all surfaces along the retraction have areas no greater than $\Sigma\setminus B_{R^2}$.   Thus by continuity, the surfaces $\Lambda_{h,R}$ can also be retracted to $\mathcal{G}$.  The area may increase slightly along the way but only by an amount depending on $h$ and which can be made arbitrarily small by shrinking $h$.  In light of \eqref{openhole} (since $2|\Sigma|$ exceeds $|\Lambda_{h,R}|$ by a fixed amount independent of $h$), by potentially shrinking $h$ further, the sweepout $\{\Lambda_{h,t}\}_{t=0}^R$ can be extended to obtain the desired sweepout of $T_{h\phi}(\Sigma)$.  Let us give more details.

For $s\in [0,1]$ define the following surfaces
\begin{align}
\Lambda_{h(s),R,s}&=\{\exp_{R_s(x)}(\pm h(s)\phi_R(x)) : x\in\Sigma\setminus B_{R^2}\}\\
& =\{\exp_{u}(\pm h(s)\phi_R(R_s^{-1}(u)) : u\in R_s(\Sigma\setminus B_{R^2})\}, 
\end{align}
where $h:[0,1]\rightarrow [0,h]$ is a non-negative function satisfying $h(0)=h$ and $h(1)=0$ and decreasing so fast so that for all $s\in [0,1]$ and some $C>0$,
\begin{equation}\label{pickh}
\|h(s)\nabla(\phi_R\circ R_s^{-1})\|_{L^\infty(R_s(\Sigma\setminus B_{R^2}))}\leq Ch\|\nabla\phi_R\|_{L^\infty(\Sigma\setminus B_{R^2})}.
\end{equation}
Combining \eqref{pickh} with \eqref{witherror} we obtain for $s\in [0,1]$
\begin{align}
|\Lambda_{h(s),R,s}| &\leq 2|R_s(\Sigma\setminus B_{R^2})|+C'h^2\|\nabla\phi_R\|^2_{L^\infty(\Sigma\setminus B_{R^2})}+C'h^2\\
&\leq 2|\Sigma\setminus B_{R^2}|+C'h^2\|\nabla\phi_R\|^2_{L^\infty(\Sigma\setminus B_{R^2})}+C'h^2\\\label{dd}
&\leq |\Lambda_{h,R}|+\epsilon(h)+C'h^2\|\nabla\phi_R\|^2_{L^\infty(\Sigma\setminus B_{R^2})}+C'h^2,\
\end{align}
where $C'$ only depends on $\Sigma$ and $\|\phi\|_{L^\infty}$.  In the last line we have used \eqref{areaclose}.  In the second line we used \eqref{complementofballs}.  Shrinking $h$ yet again so that
\begin{equation}\label{smallyet}
\epsilon(h)+C'h^2\|\nabla\phi_R\|^2_{L^\infty(\Sigma\setminus B_{R^2})}+C'h^2\leq 2|B_{R^2}|,
\end{equation}
we obtain by combining \eqref{dd}, \eqref{smallyet}, with \eqref{controlledfoliation} that for all $s\in [0,1]$,
\begin{equation}
|\Lambda_{h(s),R,s}|\leq 2|\Sigma|-\frac{\lambda}{32}h^2.
\end{equation}

Thus we can define the claimed sweepout via concatenation:
\begin{equation*}
    \Lambda_s = \begin{cases}
               \Lambda_{h,2sR}      & 0\leq s\leq 1/2\\
               \Lambda_{h(2s-1),R,2s-1}          & 1/2\leq s\leq 1.
           \end{cases}
\end{equation*}
Item (4) holds with $\tau:=\frac{\lambda}{32}$ and where $\epsilon_0$ is the final shrunken value of $h$.  Items (1), (2) and (3) follow by the construction.

Finally, to verify the last claim, if we assume in addition that $M$ has positive \text{Ric}ci curvature, set $\phi=1$.  While the function $\phi$ may not be an eigenfunction of the Jacobi operator, it still gives a direction to decrease area. Indeed, we have $$\int_\Sigma 1L1 = \int_\Sigma |A|^2+\text{Ric}(N,N)=\gamma.$$ for some $\gamma>0$.  Thus we still obtain \eqref{eigenvalue} with $\lambda=\gamma$.  The rest of the argument then follows verbatim.
\end{proof}

\section{Applications of the catenoid estimate}
In the following we will consider closed $3$-manifolds and various canonical sweepouts arising in different geometric situations.  In each case the main issue is to rule out multiplicity.  The catenoid estimate will enable us to foliate a neighborhood of an unstable minimal surface $\Sigma$ symmetrically about $\Sigma$ so that all areas are strictly less than twice the area of the minimal surface.  Since one can construct sweepouts with all areas below $2\Sigma$, one can avoid $2\Sigma$ (and higher multiples) as a min-max limit.

Let us first introduce the min-max notions we will need in the applications.  Beginning with a genus $g$ sweepout $\{\Sigma_t\}$ (as defined in Section \ref{catenoidinthree}) we need to construct comparison sweepouts which agree with $\{\Sigma_t\}$ on $\partial I^n$.  We call a collection of sweepouts $\Pi$ \emph{saturated} if it satisfies the following condition:  for any map $\psi\in C^\infty(I^n\times M,M)$ such that for all $t\in I^n$, $\psi(t,.)\in \mbox{Diff}_0(M)$ and $\psi(t,.)=id$ if $t\in\partial I^n$, and a sweepout $\{\Lambda_t\}_{t\in I^n}\in\Pi$ we have $\{\psi(t,\Lambda_t)\}_{t\in I^n}\in\Pi$.   Given a sweepout $\{\Sigma_t\}$, denote by $\Pi=\Pi_{\{\Sigma_t\}}$ the smallest saturated collection of sweepouts containing $\{\Sigma_t\}$.  We will call two sweepouts \emph{homotopic} if they are in the same saturated family. 
We define the \emph{width} of $\Pi$ to be
\begin{equation}
W(\Pi,M)=\inf_{\{\Lambda_t\}\in\Pi}\sup_{t\in I^n} |\Lambda_t|.
\end{equation}
A \emph{minimizing sequence} is a sequence of sweepouts $\{\Sigma_t\}^i\in\Pi$ such that 
\begin{equation}
\lim_{i\rightarrow\infty}\sup_{t\in I^n}|\Sigma_t^i|=W(\Pi,M).
\end{equation}
A \emph{min-max sequence} is then a sequence of slices $\Sigma_{t_i}^i$, $t_i\in I^n$ taken from a minimizing sequence so that $|\Sigma_{t_i}^i|\rightarrow W(\Pi,M)$.
The main point of the Min-Max Theory of Simon-Smith \cite{SS} (adapting the more general setting of Almgren-Pitts \cite{P} to smooth sweepouts) is that if the width is greater than the maximum of the areas of the boundary surfaces, then some min-max sequence converges to a minimal surface in $M$:

\begin{theorem} (Min-Max Theorem)\label{minmax}
Given a sweepout $\{\Sigma_t\}_{t\in I^n}$ of genus $g$ surfaces, if
\begin{equation}
W(\Pi,M)> \sup_{t\in\partial I^n} |\Sigma_t|,
\end{equation}
then there exists a min-max sequence $\Sigma_i:=\Sigma_{t_i}^i$ such that
\begin{equation}
\Sigma_i\rightarrow\sum_{i=1}^k n_i\Gamma_i \mbox{     as varifolds,} 
\end{equation}
where $\Gamma_i$ are smooth closed embedded minimal surfaces and $n_i$ are positive integers. Moreover, after performing finitely many compressions on $\Sigma_i$ and discarding some components, each connected component of $\Sigma_i$ is isotopic to one of the $\Gamma_i$ or to a double cover of one of the $\Gamma_i$, implying the following genus bounds:
\begin{equation}\label{genusbound}
\sum_{i\in\mathcal{O}}n_ig(\Gamma_i)+\frac{1}{2}\sum_{i\in\mathcal{N}}n_i(g(\Gamma_i)-1)\leq g.
\end{equation}
Here $\mathcal{O}$ denotes the subcollection of $\Gamma_i$ that are orientable and $\mathcal{N}$ denotes those $\Gamma_i$ that are non-orientable, and where $g(\Gamma_i)$ denotes the genus of $\Gamma_i$ if it is orientable, and the number of crosscaps that one attaches to a sphere to obtain a homeomorphic surface if $\Gamma_i$ is non-orientable.
\end{theorem}

The existence and regularity for smooth sweepouts were proven by Simon-Smith \cite{SS} (see \cite{CD} for a survey). Some genus bounds were proven by  De Lellis-Pellandini \cite{DP} and improved to the inequality above by the first named author \cite{Ke1}. The details of the multiparameter case of Simon-Smith theory can be found in the appendix of \cite{CGK}.
 
We will also need the following equivariant version of Theorem \ref{minmax} from \cite{Ke2} (as announced by Pitts-Rubinstein \cite{PR1} \cite{PR2}).  Let $G$ be a finite group acting on $M$ so that $M/G$ is an orientable orbifold without boundary (i.e., we exclude reflections).   A set $\Sigma\subset M$ is called $G$-equivariant if $g(\Sigma)=\Sigma$ for all $g\in G$.  Denote by $\mathcal{S}$ the set of points in $M$ where $gx=x$ for some $g$ not equal to the identity in $G$.  For $x\in\mathcal{S}$, the \emph{isotropy subgroup} $G_x$ is the set of all $g$ so that $gx=x$. Suppose $\{\Sigma_t\}_{t=0}^1$ is a one parameter genus $g$ sweepout of $M$ by $G$-equivariant surfaces so that each surface with positive area intersects $\mathcal{S}$ transversally.  Consider the saturation $\Pi^G_{\Sigma_t}$ of the family $\{\Sigma_t\}$ by isotopies through $G$-equivariant surfaces.  Then we have the following
\begin{theorem}(\cite{Ke2})\label{equivariant}
If $W(\Pi^G_{\Sigma_t})>0$, then some min-max sequence converges to a $G$-equivariant minimal surface in $M$.  The genus bound \eqref{genusbound} also holds and furthermore any compression must be $G$-equivariant.   Moreoever, a component of the min-max limit can only contain a segment of $\mathcal{S}$ that has $\mathbb{Z}_2$ isotropy.  In this case, such a component has even multiplicity.
\end{theorem}
\subsection{Min-max minimal surfaces arising from Heegaard splittings}
Let $M$ be a closed orientable $3$-manifold.  Recall that a \emph{Heegaard surface} is an orientable closed embedded surface $\Sigma\subset M$ so that $M\setminus\Sigma$ consists of two open handlebodies.  The \emph{Heegaard genus} of $M$ is the smalllest genus $g$ realized by a Heegaard splitting.  A \emph{one-sided Heegaard surface} is an embedded non-orientable surface $\Sigma\subset M$ so that $M\setminus\Sigma$ consists of a single handlebody. 

If $M$ has positive Ricci curvature and does not admit non-orientable surfaces, then \cite{MN} prove (Theorem 3.4) that a surface realizing the Heegaard genus can be isotoped to be minimal and have index $1$.  If $M$ contains embedded non-orientable surfaces, however, one could not rule out that the min-max sequence converges with multiplicity two to a one-sided Heegaard splitting surface. Using the catenoid estimate, we can rule out this possibility and thus obtain:
\begin{theorem}\label{bounds}
Let $M$ be a closed  $3$-manifold with positive Ricci curvature and $\Gamma$ a Heegaard surface realizing the Heegaard genus of $M$.  Then $\Gamma$ is isotopic to an index $1$ minimal surface $\Sigma$.   Moreover, $\Sigma$ realizes the min-max width obtained from considering saturations by Heegaard sweepouts relative to $\Gamma$.
\end{theorem}
 
\begin{remark}  As a simple example of Theorem \ref{bounds}, consider $\mathbb{RP}^3$ with its round metric.  The projection of the Clifford torus in $\mathbb{S}^3$ is a minimal torus of Morse index $1$ and area $\pi^2$ (by Theorem 3 in \cite{DRR} it is in fact the unique embedded index $1$ minimal surface).  The area of $\mathbb{RP}^2\subset\mathbb{RP}^3$ is $2\pi$.   The projected Clifford torus has smaller area than twice the area of the projective plane and when one runs a min-max procedure using Heegaard tori as sweepouts, one obtains the Heegaard torus and not the projective plane with multiplicity two.  
\end{remark}

\begin{proof}
Assume without loss of generality that $M$ is not diffeomorphic to the three-sphere (as this case is handled in Theorem 3.4 in \cite{MN}).   Denote by $\Sigma$ the support of the min-max limit arising from sweepouts in the saturation of a Heegaard foliation of $M$ by surfaces isotopic to $\Gamma$.  By Frankel \cite{F}, $\Sigma$ is connected. Let us assume toward a contradiction that $\Sigma$ is a non-orientable minimal surface.  We will construct a Heegaard sweepout of $M$ by surfaces isotopic to the original Heegaard surface with all areas strictly less than $2|\Sigma|$, contradicting the definition of width.

The following claim is a consequence of the surgery process of \cite{Ke1} together with  topological arguments that use   the fact that $\Gamma$ is strongly irreducible.  This is based on Lemma 1.6 in \cite{S}.
\\
\\ \noindent
\emph{claim: For $\epsilon$ small enough, $\Gamma$ is isotopic to $\partial T_\epsilon(\Sigma)$ with a verticle handle attached.}  \\ 
\\ \indent
If  $\Sigma=\mathbb{RP}^2$, then by Frankel's theorem \cite{F}, $M=\mathbb{RP}^3$.  But the unique genus $1$ Heegaard splitting of $\mathbb{RP}^3$ is obtained by attaching a verticle handle to $\partial T_{\epsilon}(\Sigma)$, so the claim is established in this case.

Since $M$ has positive Ricci curvature, it contains no incompressible two-sided minimal surfaces.  Thus by Casson-Gordon \cite{CG}, since $\Gamma$ is a lowest genus Heegaard surface, it must be a strongly irreducible Heegaard splitting.  From Theorem 1.9 in \cite{Ke1} we know that after surgeries the min-max sequence is isotopic to  $\cup_{i=1}^k\partial T_{\epsilon_i}(\Sigma)$ for some increasing set of numbers $\epsilon_1, ...,\epsilon_k$.  Because $\Gamma$ is strongly irreducible, surgeries along essential curves have to be performed in the same side of $\Gamma$. Surgeries along non-essential curves can occur on both sides and split off spheres. 

If $\Sigma\neq \mathbb{RP}^2$, then no $\partial T_{\epsilon_i}(\Sigma)$, $i=1,\ldots,k$ is a sphere and so they had to be obtained from surgeries performed  in the same side of $\Gamma$, which means they all bound handlebodies with disjoint supports. Thus $k>1$ forces the handlebody of  $\partial T_{\epsilon_1}(\Sigma)$ to contain $\Sigma$, which is impossible. Thus $k=1$.

By irreducibility,  $\Gamma$ is obtained from $\partial T_{\epsilon_1}(\Sigma)$ by adding a single verticle handle through $\Sigma$ (see \cite{H}).

 Thus the claim is established.  
 
 We now will construct a Heegaard sweepout of $M$ by surfaces isotopic to $\Gamma$
with all areas less than $2|\Sigma|$.

There is a double cover of $M$, $\tilde{M}$ (also with positive \text{Ric}ci curvature) so that the projection map $\pi:\tilde{M}\rightarrow M$ is a local isometry and so that $\tilde{\Sigma} := \pi^{-1}(\Sigma)$ is an orientable Heegaard surface in $\tilde{M}$. Moreoever, $M=\tilde{M}/\{1,\tau\}$, where $\tau:\tilde{M}\rightarrow\tilde{M}$ is an involution switching the two handlebodies determined by $\tilde{\Sigma}$. Since $\tilde{M}$ has positive Ricci curvature, $\tilde{\Sigma}$ is also unstable.  Hence by \cite{MN} (Lemma 3.5), we can find an optimal sweepout $\{\Sigma_t\}_{t=0}^{1/2}$ of one of the handlebodies $H_1$ in $\tilde{M}$ bounded by $\tilde{\Sigma}$ by surfaces isotopic to $\tilde{\Sigma}$ in the sense that:
\begin{enumerate}
\item $|\Sigma_t| < |\tilde{\Sigma}|$ for all $t\in [0,1/2]$,
\item for $t$ near $1/2$, $\Sigma_t$ coincides with $\{\exp_p((1/2-t)N(p)):p\in\tilde{\Sigma}\}$ (i.e. parallel surfaces to $\tilde{\Sigma}$, where $N$ is the unit normal on $\tilde{\Sigma}$ pointing into $H_1$.
\end{enumerate}

Fix a point $p\in\tilde{\Sigma}$ and also consider image $q:=\tau(p)\in\tilde{\Sigma}$.  For each $t\in [0,1/2]$ let us choose two disinct points $p_t$, $q_t$ in $H_1$ (varying smoothly in $t$) so that 
\begin{enumerate}
\item $p_t$ and $q_t$ are both contained in $\Sigma_t$ for $t\in [0,1/2]$
\item for $t$ close to $1/2$, $p_t =\exp_p((1/2-t)N(p))$ 
\item for $t$ close to $1/2$, $q_t = \exp_q ((1/2-t)N(q))$
\end{enumerate}

Also for $t\in [0,1/2]$, choose arcs $\alpha_t$ and $\beta_t$ in $H_1$ varying smoothly with $t$ so that
\begin{enumerate}
\item  $\alpha_t$ begins at $p_t$ and ends at $p$ and $\beta_t$ similarly joins $q_t$ to $q$.
\item  $\alpha_t\cap\Sigma_t= p_t$ and $\beta_t\cap\Sigma_t= q_t$, 
\item  for $t$ close to $1/2$, $\alpha_t$ (resp. $\beta_t$) consists of  the normal geodesic to $\tilde{\Sigma}$ from $p$ (resp. $q$) to
$\Sigma_t$,
\item  for $t=1/2$, $\alpha_t= p$ and $\beta_t=q$. 
\end{enumerate}

For any $\epsilon>0$, let us set $D^1_{t,\epsilon} := T_\epsilon(\alpha_t)\cap\Sigma_t$ and $D^2_{t,\epsilon} := T_\epsilon(\beta_t)\cap\Sigma_t$.  There exists $\epsilon_0$ so that whenever $\epsilon<\epsilon_0$ we have that for all $t$, $D^1_{t,\epsilon}$ and $D^2_{t,\epsilon}$ are both disks.

Finally we can make a new sweepout $\{\Gamma_t\}_{t=0}^{1/2}$ by gluing in tubes and removing the appropriate disks:
\begin{equation}
\Gamma_t=\Sigma_t\cup \partial(T_{\eta(t)}(\alpha_t))\cup T_{\eta(t)}(\beta_t)\setminus(D^1_{t,\eta(t)}\cup D^2_{t,\eta(t)}),
\end{equation}
\noindent
where $\eta(t):[0,1/2]\rightarrow [0,\epsilon_0]$ and satisfies

\begin{enumerate}
\item $\eta(0)=0$, 
\item $\eta(t)>0$ for $t\in (0,1/2-\delta)$ 
\item $\eta(t)=0$ for $t\in [1/2-\delta,1/2]$.  
\end{enumerate}

The parameter $\delta$ will be chosen later.  Then we can consider the sweepout for $t\in [0,1/2]$ given by 
\begin{equation}
\tilde{\Gamma}_t = \Gamma_t\cup\tau(\Gamma_t).
\end{equation}

We can now apply the Catenoid Estimate (Theorem \ref{catenoidinthreemanifold}) to replace the sweepout $\{\tilde{\Gamma}\}_{t=0}^{1/2}$ since it coincides with parallel surfaces to $\tilde{\Sigma}$ for $t$ near $1/2$.   To that end we set $p_1=p$ and $p_2=q$ and let $\mathcal{G}$ be a graph onto which $\tilde{\Sigma}\setminus(p\cup q)$ retracts $\tau$-equivariantly.  We can then choose $\delta$ to be smaller than the $\epsilon_0$ provided by the Catenoid Estimate. 

In this way we obtain a new $\tau$-equivariant sweepout $\Gamma'_t$ of $\tilde{M}$ so that $|\Gamma'_t| < 2|\tilde{\Sigma}|$ for all $t\in [0,1/2]$).  The surfaces $\Gamma'_t$ projected down to $M$ then give rise to a Heegaard foliation of $M$ (isotopic to $\Gamma$) with all areas strictly less than $2|\Sigma|$.   This contradicts the definition of min-max width.  Thus $\Sigma$ cannot be non-orientable.

It remains to rule out that other degeneration has occurred, i.e., that $g(\Sigma) < g(\Gamma)$.  In this case, however, $\Sigma$ would be a Heegaard splitting of smaller genus than $\Gamma$, contradicting the assumption that $\Gamma$ realized the Heegaard genus of $M$.  Thus $\Sigma$ and $\Gamma$ are isotopic.  The index bound follows from Lemma 3.5 in \cite{MN}.

\end{proof}

Adapting these ideas to the setting of Almgren-Pitts, we obtain:

\begin{theorem}
The Almgren-Pitts width (with $\mathbb{Z}$ or $\mathbb{Z}_2$ coefficients) of an orientable $3$-manifold $M$ with positive \text{Ric}ci curvature is achieved by an index $1$ orientable minimal surface.
\end{theorem}
\begin{proof}
If the theorem failed, then by \cite{Z} Theorem 1.1 (see also \cite{MR}) the width of $M$ must be realized by a non-orientable minimal surface $\Sigma$ with multiplicity two.  Again we can lift $\Sigma$ to $\tilde{\Sigma}$ in a double cover of $M$, $\tilde{M}$, so that $M=\tilde{M}/\{1,\tau\}$.  Since $\tilde{M}$ has positive Ricci curvature, $\tilde{\Sigma}$ is an unstable Heegaard splitting.  Thus again by Lemma 3.4 in \cite{MN} we can extend $\tilde{\Sigma}$ to a sweepout $\{\tilde{\Sigma}_t\}_0^1$ of $\tilde{M}$ where $\tilde{\Sigma}=\tilde{\Sigma}_{1/2}$ and with $|\tilde{\Sigma}_{t}|< 2|{\Sigma}|$ for all $t\neq 1/2$.  Moreoever, for $t$ close to $1/2$ the surfaces $\{\tilde{\Sigma}_t\}$ are a foliation of a neighborhood  $T_\epsilon(\tilde{\Sigma})$ of  $\tilde{\Sigma}$ by parallel hypersurfaces (for suitably small $\epsilon$).  For $\epsilon$ small enough,  we can then apply the catenoid estimate to $T_\epsilon(\tilde{\Sigma})$ where necks are added at some $p\in\tilde{\Sigma}$ and $\tau(p)\in\tilde{\Sigma}$.  We thus obtain from $\{\tilde{\Sigma}_t\}_0^1$ a $\tau$-equivariant sweepout of $\tilde{M}$ continuous in the flat topology with all areas strictly less than $2\Sigma$.   Such a family is continuous in the $\mathbf{F}$-metric and thus admissible in the sense of Almgren-Pitts (see \cite{MN4}).   The projected sweepout to $M$ consist of surfaces with all areas strictly below $2|\Sigma|$, contradicting the assumption that $2\Sigma$ realized the width of $M$.
\end{proof}

\subsection{Doubling constructions}
We will give a min-max construction of the doubled Clifford torus due to Kapouleas-Yang \cite{KY}.  

We first identify $\mathbb{S}^3$ with a subset of $\mathbb{C}^2\equiv\mathbb{R}^4$:
$$\mathbb{S}^3=\{(z,w)\in\mathbb{C}^2\;|\;|z|^2+|w|^2=1\}.$$
Consider the group $H_m=\mathbb{Z}_{m}\times\mathbb{Z}_{m}$ acting on $\mathbb{S}^3$ as follows.   For any $([k],[l])\in \mathbb{Z}_{m}\times\mathbb{Z}_{m}$ and $(z,w)\in\mathbb{S}^3$, we define the action to be 
\begin{equation}
([k],[l]).(z,w)=(e^{2\pi k i/m}z,e^{2\pi il/m}w).
\end{equation}

Let $G_m$ be the group of order $2m^2$ generated by both $H_m$ and the involution $\tau:\mathbb{S}^3\rightarrow\mathbb{S}^3$ given by
\begin{equation}
\tau(z,w)=(w,z).
\end{equation}

Fix an integer $m\geq 2$.  We can then consider $G_m$-equivariant sweepouts of $\mathbb{S}^3$ of genus $m^2+1$.  The surfaces in our sweepout consist of two tori, one on each side of the Clifford torus, joined by $m^2$ tubes.   Precisely, we can start with the foliation of $\mathbb{S}^3$ by constant mean curvature surfaces $\{\Gamma_t\}_{t=0}^1$
\begin{equation}
\Gamma_t=\{(z,w)\in\mathbb{S}^3 : |z|^2= t\}.
\end{equation}

Then consider the $G_m$-equivariant family $\{\Gamma'_t\}_{t=0}^{1/2}$
\begin{equation}
\Gamma'_t=\Gamma_t\cup\tau(\Gamma_t).
\end{equation}

Note that $\Gamma'_0$ consists of the two circles $\{w=0\}$ and $\{z=0\}$ and $\Gamma'_{1/2}$ consists of the Clifford torus counted with multiplicity two.  We now must connect the two components of $\Gamma'_t$ by $m^2$ $G_m$-equivariant necks to obtain a sweepout of $\mathbb{S}^3$ by surfaces of genus $m^2+1$.  This is straightforward but we include the details.  

Let us parameterize the Clifford torus $C$ as $(z,w)=(\frac{1}{\sqrt{2}}e^{i\theta},\frac{1}{\sqrt{2}}e^{i\phi})$ for $\theta\in [0,2\pi]$ and $\phi\in [0,2\pi]$. In this way the Clifford torus $C=\Gamma_{1/2}$ can be thought of as a square grid $[0,2\pi]\times [0,2\pi]$ in $(\theta, \phi)$-coordinates, where the obvious identifications on the boundary of $[0,2\pi]\times [0,2\pi]$ result in a torus. 

Let us define $\phi_k$ for $k\in\{1,2,..m\}$ to be the line in the torus given by $\phi=\frac{2k\pi i}{m}$.  Similarly we set $\theta_j$ for $j\in\{1,2,..m\}$ to be the line $\theta=\frac{2j\pi i}{m}$.   Consider the $m^2$-squares $\{S_i\}_{i=1}^{m^2}$ into which the lines $\{\theta_i\}_{i=1}^m$ and $\{\phi_j\}_{j=1}^m$ divide $[0,2\pi]\times [0,2\pi]$.   The action $G_m$ restricted to $C$ has  as fundamental domain half of a square, i.e. a triangle $T_i$ in $S_i$ (cut through either diagonal).  Indeed, the action of $H_m$ on $C$ has any square $S_i$ as fundamental domain, but since $G_m$ includes the involution that maps the lines $\theta_*$ to $\phi_*$, and vice versa, a fundamental domain for the action of $G_m$ is cut in half.  Thus a fundamental domain for the action of $G_m$ on $\mathbb{S}^3$ is a polyhedron over the square $S_i$ (containing one triangle of $S_i$).

The necks will be added at the centers $C_i$ of the $m^2$ squares $\{S_i\}_{i=1}^{m^2}$ in the grid.  

Let us set
\begin{equation}
\mathcal{G} = \bigcup_{i=1}^m\phi_i\cup \bigcup_{j=1}^m\theta_j.
\end{equation}

Note that $C\setminus\{C_i\}_{i=1}^{m^2}$ retracts onto $\mathcal{G}$ and the retraction can be performed $G_m$-equivariantly.  

Let us define for $t\in [0,1/2]$ the path
\begin{equation}
\alpha(t)= (\sqrt{1-t} e^{\pi i /m},\sqrt{t} e^{\pi i /m}).
\end{equation}

The curve $\alpha(t)$  connects the point $(e^{\pi i /m},0)\in\mathbb{S}^3$ in the circle $\{w=0\}$ to the point $(\frac{1}{\sqrt{2}}e^{\pi i /m},\frac{1}{\sqrt{2}}e^{\pi i /m})$ in $C$ (i.e. one of the center points $C_i$).  For $\epsilon>0$ denote by $T_\epsilon(\alpha(t))$ the tubular neighborhood about $\alpha(t)$.  Note that 
\begin{equation}\label{tubular}
|\partial(T_\epsilon(\alpha([0,1/2])))|\rightarrow 0 \text{ as } \epsilon\rightarrow 0.
\end{equation}

Denote the disk (for $\epsilon$ small enough, depending on $t$):
\begin{equation}
D_{t,\epsilon} = T_\epsilon(\alpha([0,1/2]))\cap\Gamma_t.
\end{equation}

Choose a smooth function $\eta:[0,1/2]\rightarrow [0,\epsilon]$, so that $\eta(t)=0$ for $t\geq 1/2-\delta$ and $\eta(0)=0$ (where $\epsilon>0$ and $\delta>0$ will be chosen later).  The function $\eta$ will determine the thickness of the tubes added. 

We can now finally define our amended sweepouts, where tubes have been added and their attaching disks removed:
\begin{equation}
\Gamma''_t=\Gamma'_t\cup\bigcup_{g\in G_m} g(\partial T_{\eta(t)}(\alpha(t))\setminus\bigcup_{g\in G_m} g(D_{t,\eta(t)}).
\end{equation}

In light of \eqref{tubular} and the Catenoid Estimate (applied with $\Sigma=C$, the graph $\mathcal{G}$ and points the union of the $C_i$), it is clear that for $\epsilon$ and $\delta$ chosen appropriately, by adjusting $\Gamma''_t$ near the Clifford torus we can obtain a new family $\Gamma'''_t$,  satisfying
\begin{equation}\label{lessthantwice}
\sup_{t\in [0,1/2]} |\Gamma'''_t| < 2|C| = 4\pi^2.
\end{equation}
 
Except for one value of $t\in (0,1)$, $\Gamma'''_t$ is a genus $m^2+1$ piecewise smooth surface.   By small perturbation, we can then obtain an $G_m$-equivariant smooth sweepout of $\mathbb{S}^3$ by genus $m^2+1$ surfaces with all areas still less than $2|C|$.

Considering all sweepouts in the equivariant saturation $\Pi^{G_m}$ of our canonical family $\Gamma'''_t$, we can then define the equivariant min-max width:
\begin{equation}
\omega_1^{G_m}=\inf_{\Lambda(t)\in\Pi^{G_m}} \sup_{t\in[0,1]}|\Lambda(t)|.
\end{equation}

Theorem \ref{equivariant} then asserts the existence of a smooth embedded connected $G_m$-equivariant minimal surface 
in $\mathbb{S}^3$.  We obtain
\begin{theorem}\label{doubledtori}
For any integer $m\geq 2$, the min-max limit $\Sigma_m$ in the saturation $\Pi^{G_m}$ is an embedded $G_m$-equivariant minimal surface.  The genus of $\Sigma_m$ approaches infinity as $m\rightarrow\infty$.  Moreover $|\Sigma_m| < 4\pi^2$ and $\Sigma_m\rightarrow 2C$ in the sense of varifolds as $m\rightarrow\infty$.
\end{theorem}
\begin{remark}
In fact one can classify the possible compressions and show that the genus of $\Sigma_m$ is $m^2+1$ when $m$ is large.  See \cite{Ke2} for more details.
\end{remark}
\begin{proof}
Note that $\Sigma_m$ must occur with multiplicity $1$.  Indeed, if the multiplicity were $k>1$, then by \eqref{lessthantwice} we have $k|\Sigma_m| < 2|C|$, where $|C| = 2\pi^2$, the area of the Clifford torus. Thus $|\Sigma_m| < |C|$.  By the resolution of the Willmore conjecture \cite{MN2}, in this case $\Sigma_m$ can only be a round equator, which is not $G_m$-equivariant.  Thus the multiplicity of $\Sigma_m$ must be one.

It remains to show that $\Sigma_m\rightarrow 2C$.  Indeed suppose some subsequence (not relabeled) of $\Sigma_m$ converges to a stationary varifold $V$ different from $2C$.  By the monotonicity formula, $V$ cannot be equal to $0$.  Moreover, the support of $V$ is invariant under the limiting groups $H_\infty=\mathbb{S}^1\times\mathbb{S}^1$ and the involution $\tau$.  Thus any point in the support of $V$ not on the Clifford torus forces an entire cmc torus parallel to $C$ to be contained in the support of $V$.  None of these tori are stationary except the Clifford torus. It follows that the support of $V$ is contained on the Clifford torus $C$.  By the Constancy Theorem, $V=kC$ for some integer $k$.  By the strict upper bound of $4\pi^2$ on the equivariant widths, $k$ is equal to either $1$ or $2$.  If $k=1$, it follows that $\Sigma_m\rightarrow C$ smoothly.  This implies $\Sigma_m$ is a rotated Clifford torus for every large $m$.   But $C$ is the unique such Clifford torus invariant under $G_m$ when $m$ is large.  Thus $\Sigma_m=C$ for large $m$.  On the other hand, by Theorem \ref{equivariant}, $\Sigma_m$ is never equal to the Clifford torus $C$ with multiplicity $1$ since in this case it would contain a segment of the $\mathbb{Z}_2$-isotropy singular set and yet have odd multiplicity.  Thus $V=2C$.  It follows that $\Sigma_m\rightarrow 2C$. By the compactness theorem of Choi-Schoen \cite{CS} it follows that the genus of $\Sigma_m$ approaches infinity as $m\rightarrow\infty$.

\end{proof}

\section{Higher dimensional case}
In this section, we prove Theorem \ref{almgrenpitts} when $4\leq n+1\leq 7$, which we restate:

\begin{theorem}
For $4\leq (n+1) \leq 7$, the Almgren-Pitts width (with $\mathbb{Z}$ or $\mathbb{Z}_2$ coefficients) of an orientable  $(n+1)$-manifold $M^{n+1}$ with positive Ricci curvature is achieved by an orientable index $1$ minimal hypersurface with multiplicity $1$.
\end{theorem}  

\begin{proof}
By Theorem 1.1 in \cite{Z}, if it failed, the width of $M^{n+1}$ is achieved by a closed embedded non-orientable minimal hypersurface $\Gamma^n$ with multiplicity $2$.  Similarly to the two-dimensional case, for some finite set $\mathcal{P}\subset\Gamma^n$, there exists a retraction $R_t$ from $\Gamma^n\setminus\mathcal{P}$ to the lower dimensional skeleton of $\Gamma^n$.  Let us assume that the cardinality of $\mathcal{P}$ is $1$ as the general case follows analogously.

We can find a double cover $\tau:\tilde{M}\rightarrow M$ so that $\tau^{-1}(\Gamma^n)$ is an orientable embedded minimal hypersurface, say $\Sigma^n$ of $\tilde{M}$ and so that both components $H_1$ and $H_2$ of $\tilde{M}\setminus\Sigma^n$ are diffeomorphic to $M\setminus\Gamma^n$ (see Proposition 3.7 in \cite{Z}).   By Proposition 3.6 in \cite{Z}, for some $\epsilon>0$ and some $\delta >0$, we can construct a sweepout $\{\Sigma_t\}_0^1$ of $H_1$ so that $\Sigma_0=\Sigma^n$ with $\sup_{t\in[\delta,1]}\mathcal{H}^n(\Sigma_t^n)\leq\mathcal{H}^n(\Sigma^n)-\epsilon$ and so that for $t\in [0,\delta]$, the sweepout $\{\Sigma_t\}$ consists of surfaces parallel (via the exponential map) to $\Sigma^n$, i.e. $\Sigma_t=\{\exp_p (tN(p)) : p\in\Sigma\}$.

We now will amend the sweepout $\{\Sigma_t\}$ to produce a new sweepout $\{\Lambda_t\}$ that for some $\delta'\leq\delta$ agrees with $\{\Sigma_t\}$ for $t\in [\delta',1]$ but in the interval $[0,\delta']$ it ``opens up via cylinders'' at the points $\tau^{-1}(\mathcal{P})$.  Moreover, $\Lambda_0$ will consist of lower dimensional skeleta of $\Sigma^n$ (which have zero $n$-dimensional Hausdorff measure). 

The key point, of course, is that in addition $\{\Lambda_t\}$ will satisfy
\begin{equation}
\sup_{t\in [0,1]}\mathcal{H}^n(\Lambda_t) < \mathcal{H}^n(\Sigma^n).
\end{equation}
The projected sweepout $\tau(\Lambda_t)$ then gives rise to a sweepout of $M^{n+1}$ all of whose slices have areas strictly less than that of $2\Gamma^n$, contradicting the definition of width.  This will complete the proof of Theorem \ref{almgrenpitts}. Let us now construct the desired sweepout $\{\Lambda_t\}$ when $t\in [0,\delta]$.  It will consist of $\{\Sigma_t\}_{t=0}^1$ where a cylinder has been added and the corresponding disk from $\Sigma_t$ removed.

Because on a small enough scale, $\Sigma^n$ is roughly Euclidean, we have the following Euclidean volume comparisons.  Namely, there exists an $R>0$ so that for any $p\in\Sigma^n$ and $t\leq R$ there holds:
\begin{equation}
ct^n\leq \mathcal{H}^n(\Sigma^n\cap B_t(p))\leq Ct^n,
\end{equation}
\begin{equation}
ct^{n-1}\leq \mathcal{H}^{n-1}(\Sigma^n\cap \partial B_t(p))\leq Ct^{n-1}.
\end{equation}
Moreoever there exists $h_0>0$ so that whenever $h\leq h_0$ one has Euclidean-type area bounds for the small cylinders $$C_{t,h}(p):=\{\exp_p(tN(p)  : p\in\Sigma^n\cap\partial B_t(p)), t\in[-h,h]\}$$ diffeomorphic to $\mathbb{S}^{n-1}\times [-h,h]$ based at $\Sigma^n$.  That is, there holds
\begin{equation}\label{cylinders}
cht^{n-1}\leq \mathcal{H}^n(C_{t,h}(p))\leq Cht^{n-1}.
\end{equation}
Also for $h\leq h_0$ we have volume bounds for small balls pushed via the exponential map. Setting $B_{h,t}(p):= \{\exp_x(hN(p)) : x\in B_t(p)\cap\Sigma^n\}$, we have
\begin{equation}\label{expofballs}
ct^n\leq\mathcal{H}^n(B_{h,t}(p))\leq Ct^n.
\end{equation}

Since $\Sigma^n$ is minimal, it follows from \eqref{simple} that we have for $h\in [0,\delta]$
\begin{equation}\label{simple2}
\mathcal{H}^n(\Sigma_h)\leq \mathcal{H}^n(\Sigma^n)-Ah^2.
\end{equation}

Fix $p\in\mathcal{P}$. For $t\in [0,R]$, let us denote the amended surfaces 
\begin{equation}
\Lambda_{h,t}:= \Sigma_h\cup (C_{t,h}(p)\setminus B_{t,h}(p))\cup (C_{t,h}(\tau(p))\setminus B_{t,h}(\tau(p))).
\end{equation}

It follows from \eqref{cylinders}, \eqref{expofballs} and \eqref{simple2} that we can then estimate
\begin{equation}\label{areaa}
\mathcal{H}^n(\Lambda_{h,t})\leq \mathcal{H}^n(\Sigma^n)+2Cht^{n-1}-2ct^n-Ah^2.
\end{equation}

The maximum value of the function $2Cht^{n-1}-2ct^n$ occurs at 
\begin{equation}\label{max}
t=\frac{C(n-1)h}{cn}.
\end{equation}
Plugging \eqref{max} into \eqref{areaa} it follows that for some $B>0$ (independent of $t$)
\begin{equation}
\mathcal{H}^n(\Lambda_{h,t})\leq\mathcal{H}^n(\Sigma^n)+Bh^n-Ah^2.
\end{equation}
Shrinking $h_0$ yet again we have that for all $h\leq h_0$ and $t\leq R$, 
\begin{equation}\label{quadbound}
\mathcal{H}^n(\Lambda_{h,t})\leq \mathcal{H}^n(\Sigma^n)-\frac{A}{2}h^2.
\end{equation}
Moreoever, when $t=R$ we obtain
\begin{equation}
\mathcal{H}^n(\Lambda_{h,R})\leq \mathcal{H}^n(\Sigma^n)-2cR^n+2CR^{n-1}h - Ah^2.
\end{equation}
Shrinking $h_0$ again so that $h_0\leq\frac{c}{2C}$, we obtain for $h\leq h_0$,
\begin{equation}\label{opened}
\mathcal{H}^n(\Lambda_{h,R})\leq 2\mathcal{H}^n(\Sigma^n)-cR^n - Ah^2.
\end{equation}
Thus by ``opening the hole" up to time $t=R$ we decrease area by a definite amount depending on $R$ and not on $h$.  
Using \eqref{quadbound}, \eqref{opened} along with the retraction $R_t$ we can now argue identically as in the proof of the Catenoid Estimate to extend the sweepout $\{\Lambda_{t,h}\}_{t=0}^R$ to $\{\Lambda_{t,h}\}_{t=0}^1$ foliating the entire (positive) tubular neighborhood of $\Sigma^n$ with all areas strictly less than $\mathcal{H}^n(\Sigma^n)$.  Set $\delta'=h$.  Then since $\Lambda_{0,\delta'}$ agrees with $\Sigma_{\delta'}$, by concatenating $\{\Lambda_{t,\delta'}\}_{t=0}^1$ with $\{\Sigma_t\}_{t=\delta'}^1$ we obtain the desired sweepout.  
\end{proof}

\end{document}